\documentclass[10pt, oneside, a4paper]{article}
\usepackage{geometry}
\usepackage[english]{babel}
\usepackage[T1]{fontenc}
\usepackage[latin1]{inputenc}
\usepackage{lmodern}
\usepackage{amssymb,amsmath,amscd}  
\usepackage{mathrsfs}
\usepackage{ntheorem}

{\theoremstyle{plain}
\theoremseparator{ --}
\newtheorem{theo}{Theorem}[section]}

{\theoremstyle{plain}
\theoremseparator{ --}
\newtheorem{coro}[theo]{Corollary}}

{\theoremseparator{ --}
\newtheorem{lemm}[theo]{Lemma}}

{\theoremseparator{ --}
\newtheorem{prop}[theo]{Proposition}}

{\theoremseparator{ --}
}

{\theoremseparator{ --}
\newtheorem{itheo}{Theorem}}

{\theorembodyfont{\normalfont}
\newtheorem{defi}[theo]{Definition}}

{\theoremstyle{plain}
\theorembodyfont{\normalfont}
}

{\theoremstyle{plain}
\theorembodyfont{\normalfont}
\newtheorem{rema}[theo]{Remark}}

{\theoremheaderfont{\itshape}
\theorembodyfont{\normalfont}
\theoremseparator{: }
\newtheorem*{proof}{Proof}}

\newcommand{\ProofEnd}{\hfill $\Box$}

\usepackage{tabularx}
\usepackage{array}
\usepackage{multirow}
\usepackage{multicol}
\usepackage{enumerate}
\usepackage{graphicx}
\usepackage[table]{xcolor}	

\usepackage{fancyhdr}

\usepackage{url}

\DeclareMathOperator{\Reg}{Reg}
\DeclareMathOperator{\ord}{ord}
\DeclareMathOperator{\Gal}{Gal}

\DeclareMathOperator{\DEGRE}{deg }
\renewcommand{\deg}{\DEGRE}

\newcommand{\Q}{\ensuremath{\mathbb{Q}}}

\newcommand{\Z}{\ensuremath{\mathbb{Z}}}
\renewcommand{\P}{\ensuremath{\mathbb{P}}}
\newcommand{\F}{\ensuremath{\mathbb{F}}}
\renewcommand{\H}{\ensuremath{\mathrm{H}}}
\newcommand{\Qbar}{\ensuremath{\bar{\mathbb{Q}}}}

\newcommand{\Ecal}{\mathcal{E}}
\renewcommand{\bar}[1]{\ensuremath{\overline{#1}}}
\newcommand{\hhat}[1]{\ensuremath{\widehat{#1}}}
\newcommand{\tam}{\tau}
\renewcommand{\tt}{\mathbf{t}}
\newcommand{\ie}{\textit{i.e.}{}}
\newcommand{\cf}{\textit{cf.}{}}

\input cyracc.def
\DeclareFontFamily{U}{russian}{}
\DeclareFontShape{U}{russian}{m}{n}
	{ <5><6> wncyr5
	<7><8><9> wncyr7
	<10><10.95><12><14.4><17.28><20.74><24.88> wncyr10 }{}
\DeclareSymbolFont{Russian}{U}{russian}{m}{n}
\DeclareSymbolFontAlphabet{\mathcyr}{Russian}
\makeatletter
\let\@math@cyr\mathcyr
\renewcommand{\mathcyr}[1]{\@math@cyr{\cyracc #1}}
\makeatother
\newcommand{\sha}{\mathcyr{SH}}

\newcommand{\ferm}{\ensuremath{\mathcal{F}}}

\newcommand{\cond}{\mathcal{N}}
\newcommand{\type}[1]{\mathbf{#1}}

\newcommand{\partint}[1]{\left\lfloor#1\right\rfloor}

\newcommand{\tors}{_{\mathrm{tors}}}
 
\usepackage{bm}
\usepackage{bbm}
\newcommand{\trivcar}{\mathbbm{1}}
\newcommand{\ja}{\mathbf{j}}
\newcommand{\Ja}{\mathbf{J}} 
\newcommand{\Mm}{\mathbf{m}}

\newcommand{\bA}{\mathbf{A}}
\newcommand{\ba}{\mathbf{a}}

\newcommand{\cchi}{\hm{\chi}}
\newcommand{\norm}{\mathbf{N}}
\newcommand{\gP}{\mathfrak{P}}
\newcommand{\BS}{\mathfrak{Bs}}
\newcommand{\ELL}{\mathscr{E}\!\ell\!\ell}

\usepackage[colorlinks=true,
citecolor= olive!60, 
linkcolor=black!60,
urlcolor=blue!50, 
pagebackref
]{hyperref}

\title{Explicit $L$-functions  and a Brauer-Siegel theorem  \\
for  Hessian elliptic curves}
\author{ {\large Richard Griffon}%
\thanks{E-mail: \href{mailto:r.m.m.griffon@math.leidenuniv.nl}{\texttt{r.m.m.griffon@math.leidenuniv.nl}}}
\\
{\normalsize\it Mathematisch Instituut, 
Universiteit Leiden} } 
\date{}

\renewcommand{\O}{\mathcal{O}}

\begin{document}
\pagestyle{fancy}

\maketitle

\hfill\rule{5cm}{0.5pt}\hfill\phantom{.}

\paragraph{Abstract --}
For a finite field $\F_q$ of characteristic $p\geq 5$ and $K=\F_q(t)$, we consider the family of elliptic curves  $E_d$ over $K$ given by 
$y^2+xy - t^dy=x^3$  for all integers $d$ coprime to $q$. 
We provide an explicit expression for the $L$-functions of these curves.  Moreover, we deduce from this calculation that the curves $E_d$ satisfy an analogue of the Brauer-Siegel theorem. Precisely, we show that,
for $d\to\infty$ ranging over the integers coprime with $q$, one has
\[ \log\left( |\sha(E_d/K)|\cdot \Reg(E_d/K)\right) \sim \log H(E_d/K),\]
where $H(E_d/K)$ denotes the exponential differential height of $E_d$, $\sha(E_d/K)$ its Tate-Shafarevich group and $\Reg(E_d/K)$ its Néron-Tate regulator. 


\medskip

\noindent{%

\noindent{
{\it 2010 Math. Subj. Classification:}  
11G05, 
11G40, 
14G10, 
11F67, 
11M38. 


\medskip

\hfill\rule{5cm}{0.5pt}\hfill\phantom{.}

\section*{Introduction}
\pdfbookmark[0]{Introduction}{Introduction} 
\addcontentsline{toc}{section}{Introduction}

\setcounter{section}{0}

\paragraph{} 
Let $\F_q$ be a finite field of characteristic $p\geq 5$, and $K=\F_q(t)$.
For a non-isotrivial elliptic curve $E$ over $K$, we denote by $L(E/K, T)$ its $L$-function: it is \emph{a priori} defined as a formal power series in $T$. Deep theorems of Grothendieck and Deligne, however, show that $L(E/K, T)$ is actually a polynomial with integral coefficients, 
satisfying the expected functional equation,
whose degree is given in terms of the conductor of $E$,  
and for which the Riemann Hypothesis holds.

In general, these facts are not sufficient to study finer analytic and arithmetic questions about~$E$.
For example, a general study of the distribution of zeros of $L(E/K, T)$ on the critical line appears to be out of reach at the moment. 
In the meantime, partial evidence could be gathered 
by studying special families of elliptic curves $E/K$ for which $L(E/K, T)$ is explicitly known.

\paragraph{} 
Our first goal in this article is thus to exhibit a new infinite family for which the $L$-functions can be explicitly calculated. 
Specifically, for any integer $d\geq 2$ that is coprime to $q$, consider the \emph{Hessian elliptic curve} $E_d$ over $K$, whose affine Weierstrass model is:
\begin{equation}
E_d:\quad y^2+xy - t^d \cdot y =x^3. 
\end{equation}
To give a flavour of our result (Theorem \ref{theo.Lfunc}) without having to introduce too many notations, we restrict in this paragraph to the case where $d$ divides $|\F_q^\times| = q-1$:
by cyclicity of $\F_q^\times$, we can  choose a character $\chi:\F_q^\times\to\Qbar^\times$ of exact order $d$. 
In that case, our calculation then yields:
\begin{itheo} \label{theo.i.Lfunc}
For any integer $d$ dividing $q-1$, the $L$-function of  $E_d/K$ is given by:
 \[L(E_d/K, T) = \prod_{\substack{ k =1 \\ k \neq d/2}}^{d-1}\left(1-J_k\cdot T\right)\in\Z[T],
 \quad \text{ where }
J_k := \sum_{\substack{ x_1, x_2, x_3 \in \F_{q} \\ x_1 +x_2+x_3 = 1}} {\chi}^k\left( -x_1x_2x_3\right).\]
\end{itheo}
In Theorem \ref{theo.Lfunc}, we provide a similar formula for $L(E_d/K, T)$ under the much lighter assumption that $d$ be coprime to the characteristic of $K$. In this more general setting, one has to account for the possibly nontrivial action of $\Gal(\bar{\F_q}/\F_q)$ on the $d$-th roots of unity in $\bar{\F_q}$, which leads to mild technical complications (see \S\ref{sec.prel}, \S\ref{sec.Lfunc}). 

We hope that the explicit expression for $L(E_d/K, T)$ can be of use for several applications. 
For example, using Theorem \ref{theo.Lfunc}, one could reprove a result of Ulmer stating that  as $d\geq 2$ ranges through integers coprime to $q$, the ranks of the Mordell-Weil groups $E_d(K)$ are unbounded (see \cite[\S2-\S4]{Ulmer_largeLrank}).

\paragraph{}
In \S\ref{sec.spval.bnd} and \S\ref{sec.BS},
we then use our explicit knowledge of $L(E_d/K, T)$ to prove the following asymptotic estimate (see Theorem \ref{theo.BS}): 
\begin{itheo}\label{theo.i.BS}
 Let $\F_q$ be a finite field of characteristic $p\geq 5$, and $K=\F_q(t)$. For any integer $d$ coprime with $q$, consider the Hessian elliptic curve $E_d/K$ as above. Then the Tate-Shafarevich group $\sha(E_d/K)$ is a finite group and,
as $d\to\infty$, 
\begin{equation}\label{eq.i.BS}
\log\left( |\sha(E_d/K)|\cdot \Reg(E_d/K)\right) \sim \log H(E_d/K), 
\end{equation}
where $\Reg(E_d/K)$ denotes the Néron-Tate regulator of $E_d$, and $H(E_d/K)$ its exponential differential height.
\end{itheo}
From the computation of $H(E_d/K)$, one further gets that
\[\log\left( |\sha(E_d/K)|\cdot \Reg(E_d/K)\right)  \sim \tfrac{\log q}{3} \cdot d \qquad (\text{as } d\to\infty),\]
showing that the product $|\sha(E_d/K)| \cdot \Reg(E_d/K)$ 
grows exponentially fast with $d$: see 
\cite{Hindry_MW} for an interpretation of this fact in terms of the ``complexity of computing $E_d(K)$''. 

Theorem \ref{theo.i.BS} can also be restated as:
\[\forall\epsilon>0, \qquad H(E_d/K)^{1-\epsilon} \ll_{q, \epsilon} |\sha(E_d/K)|\cdot \Reg(E_d/K) \ll_{q, \epsilon} H(E_d/K)^{1+\epsilon}.\]
The upper bound here essentially proves a conjecture of Lang (the formulation for elliptic curves over $\Q$ is \cite[Conj. 1]{Lang_conjecturesEC}). Better yet, our lower bound reveals that the exponent $1$ is optimal (\ie{} no smaller exponent would work in the upper bound).

\begin{rema} The Brauer-Siegel theorem asserts that  when $k$ runs through a sequence of number fields whose degrees over $\Q$ are bounded  and such that the absolute values $\Delta_k$ of their discriminants tend to~$+\infty$, 
 one has the asymptotic estimate
  \begin{equation}\label{eq.i.classBS}
 \log\big( |\mathrm{C}l(k)|\cdot \Reg(k) \big) \sim \log \sqrt{\Delta_k}
\qquad (\text{as } \Delta_k\to\infty), 
 \end{equation}
 where $\mathrm{C}l(k)$ denotes the class-group of $k$, and $\Reg(k)$ its regulator of units (see \cite[Chap. XVI]{LANT}).
At least in their formal structure, \eqref{eq.i.BS} and \eqref{eq.i.classBS} look very similar and, following \cite{HP15}, we view Theorem \ref{theo.i.BS} as an analogue of the Brauer--Siegel theorem for the Hessian elliptic curves.  

Note that there are only six 
infinite families of elliptic curves where a complete analogue of the Brauer-Siegel 
is known to hold unconditionally: see \cite[Thm. 1.4]{HP15}, \cite[Thm. 1.1]{Griffon_Legendre}, and four more examples in \cite{Griffon_PHD}. 

\end{rema}

Let us give a rough sketch of how we prove Theorem \ref{theo.i.BS}. 
General results of Ulmer  for elliptic curves in ``Kummer towers''  imply that for all $d$ as above, $E_d/K$ satisfies the Birch and Swinnerton-Dyer conjecture\footnote{henceforth abbreviated as BSD} (see \cite[\S6]{Ulmer_largeLrank}). In particular, the Tate-Shafarevich group $\sha(E_d/K)$ is finite and, by  bounding some of the terms appearing in the ``BSD formula'' (see Corollary \ref{coro.tamtors.Bnd}), we get
\[ \frac{\log\big(|\sha(E_d/K)|\cdot\Reg(E_d/K)\big)}{\log H(E_d/K)} 
= 1 + \frac{\log L^\ast(E_d/K, 1)}{\log H(E_d/K)} + o(1)
\qquad (\text{as } d\to\infty),\]
where $L^\ast(E_d/K, 1)$ denotes the special value of $L(E_d/K, T)$ at $T=q^{-1}$ (see Corollary \ref{coro.link.spval.BS}).
Given this link with $L^\ast(E_d/K, 1)$, proving the estimate \eqref{eq.i.BS} is equivalent to the more analytic problem of showing that 
\begin{equation}\label{eq.i.spval.bnd}
 \left| \frac{\log L^\ast(E_d/K, 1)}{\log H(E_d/K)}\right| = o(1) 
\qquad (\text{as } d\to\infty).\end{equation}
In a previous article \cite{Griffon_FermatS}, we have proved bounds on special values of   $L$-functions  of a certain type.
Since $L(E_d/K, T)$ is explicitly known, we can check that it has the correct shape to apply these results. 
The resulting upper and lower bounds are enough to ensure that \eqref{eq.i.spval.bnd} holds (see \S\ref{sec.spval.bnd}).

\paragraph{}
The paper is organized as follows.
We begin by giving, in \S\ref{sec.hessian}, a detailed presentation of the curves $E_d$ and we compute the relevant invariants: 
height, conductor, torsion subgroup and Tamagawa number.
The next two sections are devoted to the calculation of the  $L$-functions of $E_d$: \S\ref{sec.prel} introduces the necessary notations and tools  while \S\ref{sec.Lfunc} contains the result and its proof.
Finally, we show the analogue of the Brauer-Siegel theorem for $E_d$: we prove the necessary bounds on the special value  in \S\ref{sec.spval.bnd}, before recalling the BSD conjecture and concluding the proof of Theorem \ref{theo.i.BS} in \S\ref{sec.BS}.

\paragraph{Notations} 
For two functions $f(x), g(x)$ defined on $[0,\infty)$, we use both Landau's ``$f(x)= O_a(g(x))$'' and Vinogradov's ``$f(x)\ll_{a} g(x)$'' notations to mean that there exists a constant $C>0$  depending at most on the mentioned parameters $a$  such that $|f(x)|\leq C g(x)$ for $x\to\infty$.
Unless otherwise stated, all constants are effective and could be made explicit.

\numberwithin{equation}{section}

\section{Hessian elliptic curves}
\label{sec.hessian}

\begin{center}
\emph{Throughout this article, we fix a finite field $\F_q$ of characteristic $p\geq 5$,
and we denote by $K=\F_q(t)$.}
\end{center}

 Let $E/K$ be a nonconstant elliptic curve with a $K$-rational (nontrivial) $3$-torsion point $P_0$. 
 Translating $P_0$ to the origin $(0,0)$, we can assume that $E$ has an affine Weierstrass model of the form
 \[ y^2 +xy - A(t)\cdot y=x^3,\]
for some $A(t)\in\F_q[t]$ with $A(t)^3\neq 1$ (see  \cite[\S7.10]{SchShio}}). 
This model is often called the \emph{Hessian normal form} of $E$. 
In this article, we exclusively concentrate on the case when $A(t)$ is a monomial $A(t)=t^d$, for some integer $d\geq 1$ which we always assume to be coprime with $q$. 
For all such integers $d$, we thus denote by $E_d$ the elliptic curve over $K$ given by the affine Weierstrass model:
\begin{equation}\label{eq.Wmodel}
E_d:\quad y^2+xy - t^d \cdot y =x^3, 
\end{equation}
which we call the \emph{$d$-th Hessian elliptic curve} over $K$. 
It can readily be seen that the model \eqref{eq.Wmodel} has discriminant $\Delta = -t^{3d}(27t^d+1)$, and that the $j$-invariant of $E_d$ is:
\[j(E_d/K) = - \frac{(24t^d+1)^3}{t^{3d}(27t^d+1)} \in K.\]
Hence, as a map $\P^1\to\P^1$, the $j$-invariant $j(E_d/K)$ is not constant so that $E_d$ is not isotrivial. 
Note also that the $j$-invariant is separable (\ie{} $j(E_d/K)\notin K^p$) since $p\geq 5$.
  
 The reader is referred to \cite[Lecture 1]{UlmerParkCity} and \cite{SchShio} for nice expositions of basic results about elliptic curves over function fields in positive characteristic.
 
 \begin{rema} 
These elliptic curves $E_d$ have previously been studied  by Davis and Occhipinti (see  \cite{OcchipintiDavis}) from a different perspective: for many values of $d$, they have produced explicit $\F_{q^2}(t)$-rational points on $E_d$, which generate a full-rank subgroup of $E_d(\F_{q^2}(t))$. They use these explicit points to study the size of certain character sums over $\F_{q^2}$.
 \end{rema}

\subsection{Bad reduction and invariants}

Let us start by describing the bad reduction of $E_d$ and by determining the relevant invariants thereof. 

By inspection of the places of $K$ dividing the discriminant $\Delta$ of \eqref{eq.Wmodel}, 
one can see that $E_d$ has good reduction outside $\{0\}\cup B_d\cup\{\infty\}$, where $B_d$ is the set of places of $K$ that divide $27t^d+1$ 
(\ie{} $B_d$ is the set of closed points of $\P^1$ corresponding to $d$-th roots of $-1/27$). 
More precisely, we have:
\begin{prop}\label{prop.badred}
The elliptic curve $E_d$ has good reduction outside $S=\{0\}\cup B_d\cup\{\infty\}$.
The reduction of $E_d$ at places $v\in S$ is as follows:
\begin{center}\renewcommand{\arraystretch}{1.5}
\begin{tabular}{| c | c   l | c | c | c |} 
\hline
Place $v$ of $K$&  \multicolumn{2}{c|}{Type of $E_d$ at $v$ }  & $\ord_{v}\Delta_{\min}(E_d/K)$ & $\ord_{v}\cond(E_d/K)$ & $ c_v(E_d/K)$ \\
\hline \hline
$0$  & $\type{I}_{3d}$   & 
& $3d$ & $1$ & $3d$ \\ \hline
$v\in B_d$
 &  $\type{I}_1$    & & $1$ &$1$& $1$ \\  \hline
\multirow{ 3}{*}{$\infty$}
    & $\type{I}_{0}$   & if $d\equiv0\bmod{3}$ & $0$ & $0$ &  $1$ \\ \cline{2-6}
    & $\type{IV}$   & if $d\equiv-1\bmod{3}$ & $4$ & $2$ &  $3$ \\ \cline{2-6}    
    & $\type{IV}^\ast$ &  if $d\equiv-2\bmod{3}$ & $8$ & $2$ & $3$ \\ \hline
\end{tabular}
\end{center}

\noindent
In this table, for all places  $v$ of $K$ where $E_d$ has bad reduction, we have denoted by $\ord_{v}(\Delta_{\min})$ 
(resp. $\ord_v(\cond)$) the valuation at $v$ of the minimal discriminant of $E_d$ 
(resp. of the conductor of $E_d$), and  by $ c_v(E_d/K)$ the 
local Tamagawa number (see \cite[Chap. IV, \S9]{ATAEC}  for definitions). 
\end{prop}

\begin{proof}
This follows from a  routine application of  Tate's algorithm to $E_d$ 
(see \cite[Chap. IV, \S9]{ATAEC}). 
\ProofEnd\end{proof}

With this Proposition, we can compute the \emph{minimal discriminant divisor} $\Delta_{\min}(E_d/K)$ and the \emph{conductor} $\cond(E_d/K)$ of $E_d$. In particular, they have degree
\begin{equation}\label{eq.inv1}
\deg \Delta_{\min}(E_d/K) = \begin{cases}
4d &\text{ if $d\equiv0\bmod{3}$}, \\
4(d+1) &\text{ if $d\equiv-1\bmod{3}$}, \\
4(d+2) &\text{ if $d\equiv-2\bmod{3}$}, \end{cases}\quad \text{ and } \quad 
\deg \cond(E_d/K)= \begin{cases}
d+1 &\text{ if $d\equiv0\bmod{3}$}, \\
d+3 &\text{ otherwise}. \end{cases} 
\end{equation}
Indeed, note that $\sum_{v\in B_d}\deg v = d$. 
By definition, the \emph{exponential differential height of $E_d/K$} is then
\begin{equation}\label{eq.inv2}
H(E_d/K) = q^{\frac{1}{12}\deg\Delta_{\min}(E_d/K)}=q^{\partint{\frac{d+2}{3}}},
\end{equation}
where $\partint{.}$ denotes the floor function.

\begin{rema}
The following alternative definition of $H(E_d/K)$ justifies its name  (see \S2 in \cite[Lect. 3]{UlmerParkCity}). 
Let $\pi:\Ecal_d\to\P^1$ be the Néron model of $E_d/K$ and $s_0 : \P^1 \to\Ecal_d$ be the unit section, if  $\Omega^1_{\Ecal_d/\P^1}$ denotes  the sheaf of relative diffentials we let $\omega= \omega_{E_d/K}$ be the line bundle $s_0^\ast(\Omega^1_{\Ecal_d/\P^1})$ on $\P^1$. 
Then the minimal discriminant divisor $\Delta_{\min}(E_d/K)$ corresponds to a section of $\omega^{\otimes12}$. 
In particular, one has $12 \deg \omega = \deg \Delta_{\min}(E_d/K)$, and $H(E_d/K)= q^{\deg \omega }$.
\end{rema}

\begin{rema} \label{rem.minimod.loc} 
It will be convenient to have a (locally) minimal short Weierstrass model of $E_d$ at our disposal (see \S\ref{sec.Lfunc.calc}). 
By a straightforward change of variables in \eqref{eq.Wmodel}, one shows that $E_d$ can be given by:
\[ E_d:\qquad  y^2 = x^3 + x^2 -8t^d \cdot x +16 t^{2d}.\]
The discriminant of this integral Weierstrass model is $\Delta' = -2^{12}t^{3d}(27t^d+1)$. 
For all places $v\neq\infty$ of $K$,  $\ord_v \Delta' = \ord_v\Delta_{\min}(E_d/K)$ so that this new model is minimal at all the finite places $v$ of $K$. 
At $v=\infty$, the application of Tate's algorithm when $3\mid d$ (\ie{} in the case of good reduction) proves that a minimal integral model of $E$ at $\infty$ is 
$y^2+u^{d/3}xy-y=x^3$, where $u=1/t$ is the uniformizer at $\infty$. 
This model is readily brought into short Weierstrass form:
\[E_d: \quad y^2= x^3 + \frac{u^{2d/3}}{4}x^2 - \frac{u^{d/3}}{2}x + \frac{1}{4}.\]
\end{rema}

\subsection{Torsion and Tamagawa number}

In this section, we compute the torsion subgroup $E_d(K)\tors$, as well as the Tamagawa number $\tau(E_d/K)$. 

\begin{prop}\label{prop.tors} 
For any integer $d\geq 1$ coprime with $q$, one has 
\[E_d(K)\tors \simeq \Z/3\Z.\]
More precisely, $E_d(K)\tors $ is generated by $P_0=(0,0)$. 
\end{prop}

\begin{proof} 
Let $T:=E_d(K)\tors$ and $P_0=(0,0)\in E_d(K)$: it is easy to check that  $2P_0=(0,t^d) = -P_0$. 
 In particular, the point $P_0$ is $3$-torsion, and $T$ already contains a subgroup isomorphic to $\Z/3\Z$.

We observe that the $p$-part $T[p^\infty]$ of  $T$ must be trivial for $j(E_d/K)\in K$ is not a $p$-th power in $K$ (see \cite[Lect. 1, Prop. 7.3]{UlmerParkCity}). 
For any place $v$ of $K$, let $G_v$ be the component group of the fiber at $v$ of the Néron model of $E_d$ (see \cite[§7]{SchShio}, \cite[Chapter IV, §9]{ATAEC}). The table in 
\cite[p.365]{ATAEC} gives that
\begin{equation}\label{eq.comp.gps}
G_v \simeq \begin{cases}
 \Z/n\Z &\text{ if the fiber at $v$ has type $\type{I}_n$ }(n\geq 1), \\
 \Z/3\Z &\text{ if the fiber at $v$ has type $\type{IV}$ or $\type{IV}^\ast$.}
 \end{cases}
\end{equation}
We now distinguish two cases. 
First assume that $3\nmid d$: by Proposition \ref{prop.badred}, $E_d$ has additive reduction at the place $\infty$ (with type $\type{IV}$ or $\type{IV}^\ast$). 
Lemma 7.8 in \cite{SchShio} asserts that the non-$p$-part of $T$ injects into the component group $G_v$ at an additive place $v$. 
Here, this yields that  the whole of $T$ injects into $G_\infty\simeq \Z/3\Z$, and we conclude that $T\simeq\Z/3\Z$ in this case. 

We now turn to the case when $3\mid d$.
By Corollary 7.5 in \cite{SchShio}, the torsion group $T$ injects into the product $\prod_{v\mid \Delta} G_v$ of the component groups. 
Therefore, $T$ is a subgroup of $\prod_{v\mid \Delta} G_v \simeq \Z/3d\Z$ (see \eqref{eq.comp.gps}). 
From which we deduce that $T$ is cyclic of some order $M\in\Z_{\geq1}$, with  $3\mid M\mid 3d$. 

We denote by $X_1(M)$ the compactification of the modular curve classifying pairs $(E,P)$ where $E$ is an elliptic curve and $P$ is a rational point of order $M$. 
Choosing a generator  $Q\in T$, we form a pair $(E_d,Q)$ which, by construction, corresponds to a $K$-rational (non-cuspidal) point on $X_1(M)$. 
Hence, there exists a morphism $j' : \P^1\to X_1(M)$. 
As we have seen, the $j$-invariant $j(E_d/K):\P^1\to\P^1$ 
is non constant and separable, and so is the induced morphism $j'$.
Applying the Riemann-Hurwitz formula to $j'$ yields that  the genus of $X_1(M)$ has to be $0$. 
This  can only happen for  $M\in\{1,2,3,4,5,6,7,8,9,10,12\}$ (see \cite[Lect. 1, \S7]{UlmerParkCity}). 
Given that $M$ must be divisible by $3$, there remains only four possible values: $M\in\{3,6,9,12\}$. 
To conclude the proof in this case, it suffices to check  that $M$ must be odd, and that $M$ cannot be $9$. 
 
If there were a  point $P=(x,y)\in E_d(K)$ of exact order $2$, then the $x$-coordinate of $P$ would satisfy $ 4x^3-2x^2-2t^d\cdot x+t^{2d}=0$.
Letting $u=1/t$ and $x_1=u^{2d/3}\cdot x$ (recall that $3\mid d$),  we would obtain that
 $ 4x_1^3-2u^{2d/3}\cdot x_1^2-2u^{d/3}\cdot x_1+1=0$. 
But the latter equation has no solution  $x_1\in\F_q(u)$ since it factors as $(4x_1^2-2u^{2d/3}\cdot x_1-2u^{d/3})\cdot x_1=-1$.
This contradiction shows that $E_d(K)$ has no nontrivial $2$-torsion, so that $M=|T|$ is odd. 

Next suppose that there exists a $K$-rational point $Q=(x,y)$ of order exactly $9$ on $E_d$. 
Up to replacing $Q$ by one of its multiples, we can assume that 
$3\cdot Q=P_0$.
Using the triplication formula
(see \cite[Chap. III, Ex. 3.7 (d)]{AEC}), it is possible to express the $x$-coordinate of $3Q=P_0$ in terms of $x$. 
By a rather tedious computation, one can show  that $x$ must then satisfy:
\begin{equation}\notag
x^9
+6t^d\cdot x^7
+t^d(1-24t^{d})\cdot x^6
-6t^{2d}\cdot x^5
+3t^{3d}\cdot x^4
+t^{3d}(3t^{d}-1)\cdot x^3
+3t^{4d}\cdot x^2
-3t^{5d}\cdot x
+t^{6d}=0.
\end{equation}
Letting $u=1/t$, $v=u^{d/3}$ and $x_2= u^{2d/3}\cdot x = v^2\cdot x$, the above relation yields that $x_2\in\F_q(u)$ is a solution of either 
$0 =x_2^3+v\cdot x_2^2-3x_2^2-v\cdot x_2+1$ or  
\begin{align*}
0&=
x_2^6
+(3-v)\cdot x_2^5 
+(v^2+v+9)\cdot  x_2 ^4
+(v^2-3v+2)\cdot x_2 ^3
+(v^2-v+3)\cdot x_2 ^2-2v^2\cdot x_2 +1.
\end{align*}
Since none of these equations has any solution $x_2\in\F_q(u)$,  the $9$-torsion in $E_d(K)$ has to be trivial. 
 Therefore, $M = 3$ and $T\simeq \Z/3\Z$ as claimed.  
 \ProofEnd\end{proof}

The (global) \emph{Tamagawa number} $\tam(E_d/K) := \prod_{v\in S} c_v(E_d/K)$ 
can be computed from the last column of the table in Proposition  \ref{prop.badred}: we immediately get
\begin{equation}\label{eq.tam.est}
\tam(E_d/K)=\begin{cases}
3d &\text{ if $d\equiv0\bmod{3}$}, \\
9d &\text{ otherwise}. \end{cases} 
\end{equation}
In \S\ref{sec.BSD}, we will need the results of Proposition \ref{prop.tors} and \eqref{eq.tam.est} in the form of the following bound:

\begin{coro}\label{coro.tamtors.Bnd}
 For all integers $d\geq 2$, coprime with $q$, the following bound holds:
\[ \frac{\log d}{d}  \ll_q \frac{\log \left( \tam(E_d/K)\cdot q \cdot |E_d(K)\tors|^{-2}\right)}{\log H(E_d/K)} \ll_q \frac{\log d}{d} ,\qquad (\text{as }d\to\infty),
\]
for some effective constants depending at most on $q$.
\end{coro}
This is a straightforward consequence of our computations of $H(E_d/K)$, $|E_d(K)\tors|$ and $\tam(E_d/K)$.

\begin{rema}
The above Corollary could also have been obtained as a special case of deep results in \cite{HP15}. 
In that paper, 
the authors prove upper bounds on the order of the torsion subgroup (\emph{loc. cit.}, Thm. 3.8)   
and on the Tamagawa number (\emph{loc. cit.}, Thm. 6.5), 
which are valid  for abelian varieties over~$K$, under mild semistability assumptions. 
Note that their proof  is much more involved and less explicit, which is why we chose to include a self-contained treatment here. 
\end{rema}

\section[Preliminaries for the computation of the L-function]{Preliminaries for the computation of the $L$-function} 
\label{sec.prel}

The goal of the next section is to calculate the $L$-function of $E_d$ in terms of Jacobi sums. 
In this section, we introduce the necessary notations and review the required facts about characters and Jacobi sums. The notations introduced in this section will be in force for the rest of the paper.

\subsection[Action of q on Z/dZ]{Action of $q$ on $\Z/d\Z$}
\label{sec.q.act}

For any integer $d\geq 2$ coprime to $q$, $q$ acts naturally on $\Z/d\Z$ by $n\mapsto q\cdot n$. 
For any subset $Z\subset \Z/d\Z$ which is stable under this action, we denote by $\O_q(Z)$ the set of orbits of $Z$. 
In what follows, we will be particularly interested in the set
\[Z_d := \begin{cases}\Z/d\Z \smallsetminus\{0, d/3, 2d/3\} &\text{ if $d\equiv 0\bmod{3}$,} \\
 \Z/d\Z \smallsetminus\{0\} &\text{ otherwise},
 \end{cases}\]
(which is stable under multiplication by $q$ because $\gcd(d,q)= 1$) and in the corresponding set of orbits $\O_q(Z_d)$. 
Given an orbit $\Mm\in\O_q(Z_d)$, we will often need to make a choice of representative $m\in Z_d$ of this orbit: we make the convention that orbits in $\O_q(Z_d)$ are always denoted by a bold letter ($\Mm$, $\mathbf{n}$,~...) and that the corresponding normal letter ($m$, $n$,~...) designates any choice of representative of this orbit in $Z_d$.
We also identify without comment $\Z/d\Z$ with its lift $\{0, 1, 2, \dots, d-1\}$ in $\Z$.

For any orbit $\Mm\in\O_q(Z_d)$, its length $|\Mm| = \left| \left\{ m, qm, q^2m, \dots\right\}\right|$ is equal to
\[|\Mm| = \min\left\{ n\in\Z_{\geq 1} : \ q^n m \equiv m \bmod{d}\right\},\]
which can equivalently be described as the multiplicative order of $q$ modulo $d/\gcd(d,m)$, for any $m\in\Mm$.
By construction of the multiplicative order, we note that, for a power $q^n$ of $q$, one has $q^n m\equiv m \bmod{d}$ if and only if $|\Mm|$ divides $n$, \ie{} if and only if $\F_{q^n}$ is a finite extension of $\F_{q^{|\Mm|}}$.

\begin{rema}\label{rem.easy.action}
 In the special case when $d$ divides $q-1$ (\ie{} when $q\equiv 1\bmod{d}$), the action of $q$ on $Z_d$ is trivial and there is a bijection between $\O_q(Z_d)$ and $Z_d$.
\end{rema}

\subsection[Characters of order dividing d]{Characters of order dividing $d$}
\label{sec.char}

Fix an algebraic closure $\bar{\Q}$ of $\Q$ and a prime ideal $\gP$ above $p$ in the ring of integers $\bar{\Z}$ of $\bar{\Q}$: 
the residue field $\bar{\Z}/\gP$ is an algebraic closure $\bar{\F_p}$ of $\F_p$. The given finite field $\F_q$ and, more generally, any finite extension thereof 
will be seen as subfields of $\bar{\F_p}$.
The reduction map $\bar{\Z}\to\bar{\Z}/\gP$ induces an isomorphism between the group $\mu_{\infty, p'}\subset\bar{\Z}^\times$ of roots of unity of order prime to $p$ and the multiplicative group $\bar{\F_p}^\times$. 
We let $\tt:\bar{\F_p}^\times\to\mu_{\infty, p'}$ be the inverse of this isomorphism, and we denote by the same letter the restriction of $\tt$ to any finite extension of $\F_q$. 

Any nontrivial multiplicative character on a finite extension of $\F_q$ is then a power of (a  restriction of)~$\tt$, \cf{} \cite[\S3.6.2]{Cohen}. The trivial multiplicative character will be denoted by $\trivcar$.

\begin{defi}
For any $m \in \Z/d\Z\smallsetminus\{0\}$ and any integer $s\geq 1$, define a character $\tt_m^{(s)} : \F_{q^{s\cdot|\Mm|}}^\times\to \Qbar^\times$~by
\[\forall x\in\F_{q^{s\cdot|\Mm|}}^\times, \quad  \tt_m^{(s)}(x) = \left(\tt\circ\norm_{q^{s\cdot|\Mm|}/q^{|\Mm|}}(x)\right)^{(q^{|\Mm|}-1)m/d} \quad \text{ and we let } \tt_m^{(s)}(0):=0.
\]
Here, $\norm_{q^{s\cdot |\Mm|} / q^{|\Mm|}} : \F_{q^{s\cdot |\Mm|}}\to \F_{q^{|\Mm|}}$ denotes the norm of the extension $\F_{q^{s\cdot |\Mm|}}/ \F_ {q^{|\Mm|}}$. 

\noindent When $s=1$, we denote $\tt_m^{(1)}$ by $\tt_m$ for short. 
Note that $\tt_m^{(s)}$ is indeed a character because $\tt$ and the norm $\norm_{q^{s\cdot |\Mm|} / q^{|\Mm|}}$ are  both multiplicative.
\end{defi}

By construction,  $\tt_m$ is a character on $\F_{q^{|\Mm|}}^\times$ and its order divides $d$: more precisely, by noting that the restriction of $\tt$ to $\F_{q^{|\Mm|}}^\times$ has exact order $q^{|\Mm|}-1$, it can be shown that the order of $\tt_m$ is exactly $d/\gcd(d, m)$.
The ``lifted character'' $\tt_m^{(s)}$ is defined on $\F_{q^{s\cdot|\Mm|}}^\times$ and has the same order as $\tt_m$  because  the norm $\norm_{q^{s\cdot |\Mm|} / q^{|\Mm|}}$ is surjective.
Moreover, the following result shows that we can thus enumerate all characters of order dividing $d$ on finite extensions of $\F_q$:

\begin{lemm} Let $d\geq 2$ be coprime to $q$, and $\F_{q^n}$ be the extension of degree $n$ of $\F_q$. 
Denote by $X(d, q^n)$ the set of nontrivial characters $\chi$ on $\F_{q^n}^\times$ such that $\chi^d=\trivcar$. 
Then
\[X(d, q^n) = \left\{ \tt_m^{(s)}, \ m\in \Z/d\Z\smallsetminus\{0\} \text{ and } s\geq 1 \text{ such that } s\cdot |\Mm| = n\right\}.\]
\end{lemm}

\begin{proof} 
The characters $\tt_m^{(s)}$ appearing on the right-hand side all belong to $X(d,q^n)$ for they are all nontrivial characters on $\F_{q^n}^\times$ with order dividing $d$. 
To prove the converse inclusion, let $d_n = \gcd(d, q^n -1)$ and $\chi_0 = \tt^{(q^n-1)/d_n}$. 
By the discussion above, $\chi_0$ is a character on $\F_{q^n}^\times$ of exact order $d_n$. 
Thus, by cyclicity of the character group of $\F_{q^n}^\times$, $\chi_0$ generates the group $X(d, q^n)\cup\{\trivcar\}$. 
Hence, for any $\chi\in X(d, q^n)$, there is a unique $k\in\Z$ such that $1\leq k<d_n$ and $\chi = \chi_0^{k} = \tt^{(q^n-1)\cdot k/d_n}$.
Let $m=kd/d_n \in\Z$ and note that  $1\leq m<d$. 
By construction, $d$ divides $m (q^n-1)$ and we have $\chi = \tt^{(q^n-1)m/d}$. 
Recall that $|\Mm|$ is the multiplicative order of $q$ modulo $d/\gcd(d,m)$: by definition, this implies that $|\Mm|$ divides $n$ (see \S\ref{sec.q.act}), so that we can write $n = s \cdot |\Mm|$ for some $s\geq 1$.  

This shows that $\chi = \tt^{(q^{s\cdot |\Mm|}-1)m/d} = \tt^{(q^{|\Mm|}-1)m/d}\circ\norm_{\F_{q^n}/\F_{q^{|\Mm|}}}  = \tt_m^{(s)}$.
\ProofEnd\end{proof}

We will actually need the following, slightly more precise, result:

\begin{lemm}\label{lemm.reindex}
 Let $d\geq 2$ be coprime to $q$, and $\F_{q^n}$ be the extension of degree $n$ of $\F_q$. Denote by $X_3(d, q^n)$ the set of $\chi \in X(d,q^n)$ such that $\chi^3\neq\trivcar$. Then
\[X_3(d, q^n) = \left\{ \tt_m^{(s)}, \ m\in Z_d \text{ and } s\geq 1 \text{ such that } s\cdot |\Mm| = n\right\}.\]
\end{lemm}

\begin{proof}
We distinguish two cases. First, if $3\nmid d$, there are no nontrivial character of order dividing $d$ whose third power is trivial (since $3$ and $d$ are coprime), so that $X_3(d, q^n) = X(d, q^n)$. 
On the other hand, $Z_d=\Z/d\Z\smallsetminus\{0\}$, and the preceding Lemma allows us to conclude in this case. 
In the remaining case when $3$ divides $d$, we have $Z_d = \Z/d\Z\smallsetminus\{0, d/3, 2d/3\}$ and $X_3(d,q^n) = X(d,q^n)\smallsetminus\{\chi : \chi^3=\trivcar\}$.
Since the order of $\tt_m^{(s)}$ for $m\in\Z/d\Z\smallsetminus\{0\}$ is exactly $d/\gcd(d,m)$, a direct inspection shows that $(\tt_m^{(s)})^3=\trivcar$ if and only if $m=d/3$ or $2d/3$.
Which gives the desired result. 
\ProofEnd\end{proof}

\begin{rema}\label{rem.easy.char}
In the special case when $d$ divides $q-1$, \ie{} when $q\equiv 1\bmod{d}$, the characters $\tt_m$ ($m\in Z_d$) are all characters of $\F_q^\times$ because $|\Mm|=1$. 
Since there are \emph{a priori} $|Z_d|$ nontrivial characters $\chi$ on $\F_q^\times$ such that $\chi^d=\trivcar$ and $\chi^3\neq\trivcar$, 
we have enumerated all possible such characters. 
\end{rema}

\subsection{Jacobi sums} 
\label{sec.jacobi}

Let $\F_r$ be a finite field of odd characteristic (in our applications, $\F_r$ will be a finite extension of $\F_q$).
We extend the (multiplicative) characters $\chi : \F_{r}^\times \to\Qbar^\times$ to the whole of $\F_{r}$ by setting $\chi(0) = 0$ if $\chi$ is not the trivial character $\trivcar$, and by $\trivcar(0)=1$.
For a character $\chi:\F_r^\times\to\Qbar^\times$ and an extension $\F_{r'}/\F_r$ of degree $s\geq 1$, whose norm is denoted by $\norm_{{r'}/r} : \F_{r'}\to\F_r$, we let $\chi^{(s)} := \chi\circ\norm_{{r'}/r}$ be the ``lifted'' character. 

To any triple of characters $\chi_1, \chi_2, \chi_3$ on $\F_{r}^\times$, we associate a Jacobi sum
\[\ja_{r}(\chi_1, \chi_2,\chi_3) := \sum_{\substack{ x_1, x_2, x_3 \in\F_{r}  \\ x_1+ x_2+x_3=1}} \chi_1(x_1)\chi_2(x_2)\chi_3(x_3).\]
Let us recall some classical facts about these sums (see \cite[\S2.5.3-\S2.5.4]{Cohen} for details and proofs). If $\chi_1$, $\chi_2$, $\chi_3$ and $\chi_1\chi_2\chi_3$ are all nontrivial, one has $|\ja_{r}(\chi_1, \chi_2, \chi_3)| = r$.
If $\chi_1$, $\chi_2$, $\chi_3$ are nontrivial but $\chi_1\chi_2\chi_3$ is trivial, the Jacobi sum ``degenerates'' to:
\begin{equation}\label{eq.jacobi.rel}
\ja_r(\chi_1, \chi_2, \chi_3) 
= -\chi_3(-1)\cdot\sum_{\substack{ x_1, x_2\in\F_{r}  \\ x_1+ x_2=1}} \chi_1(x_1)\chi_2(x_2).
\end{equation}
Jacobi sums satisfy the Hasse-Davenport relation (see \cite[\S3.7]{Cohen}): for  any finite extension $\F_{r'}/\F_r$ of degree $s$, and any characters $\chi_1, \chi_2, \chi_3$ on $\F_r^\times$, one has
\begin{equation}\label{eq.jacobi.HD}
\ja_{r'}(\chi_1^{(s)}, \chi_2^{(s)}, \chi_3^{(s)}) = \ja_r(\chi_1, \chi_2, \chi_3)^s.
\end{equation}


We finally introduce the following notation: 
\begin{defi}\label{defi.jacobi} For any $m\in Z_d$, we let
\begin{equation}\label{eq.defi.jacobi}
\Ja(m) := \ja_{q^{|\Mm|}}(\tt_m,\tt_m,\tt_m).
\end{equation}
Notice that $\Ja(m) = \Ja(q\cdot m)$ since $x\mapsto x^q$ is a bijection of $\F_{q^{|\Mm|}}$. For an orbit $\Mm\in\O_q(Z_d)$, we can thus define  $\Ja(\Mm):=\Ja(m)$ for any choice of $m\in\Mm$. 

\noindent By construction of $Z_d$, neither $\tt_m$ nor $(\tt_m)^3$ is trivial, so that  $|\Ja(\Mm)|=q^{|\Mm|}$ for all 
$\Mm\in\O_q(Z_d)$.
\end{defi}

\section[The L-function]{The $L$-function}  
\label{sec.Lfunc}

For any place $v$ of $K$, let $q_v$ be the cardinality of the residue field $\F_{q_v}$ at $v$, and denote by $(\widetilde{E_d})_v$ the reduction modulo $v$ of a minimal integral model of $E_d$ at $v$ (a plane cubic curve over $\F_{q_v}$).
By definition, the $L$-function of $E_d$ is the power series given by
\begin{equation}\label{eq.def.Lfunc}
 L(E_d/K, T) = \prod_{v\text{ good }} \left(1 - a_v \cdot T^{\deg v} + q_v \cdot  T^{2\deg v}\right)^{-1} \cdot
\prod_{v\text{ bad }} \left(1 - a_v \cdot  T^{\deg v}  \right)^{-1} \in\Z[[T]],
\end{equation}
where the products are over the places of $K$, and where $a_v := q_v +1 - |(\widetilde{E_d})_v(\F_{q_v})|$.
Remark that, when $E$ has bad reduction at $v$, $a_v$ is $0, +1$ or  $-1$ depending on whether the reduction of $E$ at $v$ is additive, split multiplicative or nonsplit multiplicative respectively. See \cite[Lect. 1, \S9]{UlmerParkCity} for more details. 
With the notations introduced in the previous section, we can now state our main result:

\begin{theo}\label{theo.Lfunc}
Let $d\geq 2$ be an integer coprime with $q$, and set
\[Z_d := \begin{cases}\Z/d\Z \smallsetminus\{0, d/3, 2d/3\} &\text{ if $d\equiv 0\bmod{3}$,} \\
 \Z/d\Z \smallsetminus\{0\} &\text{ otherwise}.
 \end{cases}
\]
The $L$-function of $E_d/K$ is given by
\begin{equation}\label{eq.Lfunc}
L(E_d/K, T)  = \prod_{\Mm \in\O_q(Z_d)} \left(1-\tt_m(-1)\Ja(\Mm) \cdot T^{|\Mm|}\right) \in\Z[T],
\end{equation}
where $\Ja(\Mm)$ is the Jacobi sum defined in \eqref{eq.defi.jacobi}.
\end{theo}

The rest of the section is devoted to the proof of this Theorem. 
Our strategy is inspired by that of \cite[Thm. 3.2.1]{Ulmer_legII}: we calculate $L(E_d/K, T)$  by  manipulations of character sums. 
Note that an alternative, more cohomological, computation could be conducted (along the lines of \cite[\S7]{Ulmer_LargeRk}, which treats a different family of elliptic curves). 
The latter approach would be less elementary, but would have the advantage of ``explaining'' the appearance of Jacobi sums in $L(E_d/K, T)$. Indeed, one would then rely on the fact that the minimal regular model $\Ecal_d\to\P^1$ of $E_d/K$ is dominated by a quotient of the Fermat surface $\ferm_d/\F_q$ of degree $d$ (see \cite{Ulmer_largeLrank} and \cite[Lect. 3, \S10]{UlmerParkCity}), whose zeta function is well-known and involves Jacobi sums.

\begin{rema} 
In the special case when $d$ divides $q-1$, the expression in \eqref{eq.Lfunc} above simplifies. 
Indeed, letting $\chi$ be a character on $\F_q^\times$ of exact order $d$, one has
\begin{equation}\label{eq.Lfunc.simple}L(E_d/K, T)  = \prod_{\substack{ 1\leq k \leq d \\ k\neq d/2}} \left(1-\chi(-1)^k\ja_{q}(\chi^k, \chi^k, \chi^k) \cdot T\right).
\end{equation}
This was stated as Theorem \ref{theo.i.Lfunc} in the introduction; 
it follows from Theorem \ref{theo.Lfunc} with Remarks \ref{rem.easy.action} and \ref{rem.easy.char}.
\end{rema}

\subsection{Character sums identities}

The proof of Theorem \ref{theo.Lfunc} requires two identities about character sums, which we first establish.

For any finite field $\F_r$ of odd characteristic, we denote by $\lambda : \F_r^\times\to\{\pm1\}$ the unique nontrivial character of order $2$ on $\F_{r}^\times$ (the ``Legendre symbol'' of $\F_{r}$), 
extended by $\lambda(0)=0$.

\begin{prop}\label{prop.charsum}
 Let $\F_r$ be a finite field of odd characteristic. For any character $\chi:\F_r^\times\to\Qbar^\times$, 
\begin{equation}\label{eq.charsum}
\sum_{z\in\F_r}\sum_{x\in\F_r} \chi(z)\cdot\lambda(x^3+x^2-8zx+16z^2) =
\begin{cases}
0 & \text{ if $\chi$ is trivial,} \\
\chi(-1)\cdot\ja_r(\chi, \chi, \chi) & \text{ otherwise.}
\end{cases}
\end{equation}
\end{prop}

\begin{proof} Let $S_r(\chi)$ denote the double sum on the left-hand side of \eqref{eq.charsum}. We first put 
$z'=4z$ in the outer sum, exhange the order of summation and split the sum according to whether $x=0$ or not: we obtain that
\begin{align*}
S_r(\chi)
&=\chi^{-1}(4) \cdot\sum_{z'\in\F_r} \chi(z')\cdot\lambda(z')^2 
+ \chi^{-1}(4)\cdot \sum_{x\neq 0} \sum_{z'\in\F_r} \chi(z') \cdot \lambda(x^3+x^2 -2z'x+{z'}^2) \\
&=\chi^{-1}(4) \cdot\sum_{z'\in\F_r^\times} \chi(z')
+ \chi^{-1}(4)\cdot \sum_{x\neq 0}\left( \sum_{z'\in\F_r} \chi(z') \cdot \lambda(x^3+(x-{z'})^2) \right)
\end{align*}
To treat the sum over $x\neq0$, note the following: for a given $x\neq0$, by writing $z' =x(y+1)$, we obtain that 
\[\sum_{z'\in\F_r} \chi(z') \cdot \lambda(x^3+(x-{z'})^2) 
= \chi(x)\cdot \sum_{y\in\F_r} \chi(y+1)\cdot \lambda(x+y^2).\]
Summing this identity over all $x\neq 0$  and exchanging the order of summation yields that
\[S_q(\chi) = \chi^{-1}(4) \cdot \sum_{z'\neq 0} \chi(z') +
\chi^{-1}(4) \cdot \sum_{y\in\F_r} \chi(y+1) \left(\sum_{x\neq 0} \chi(x) \lambda(x+y^2)\right).\]
If $\chi$ is trivial, a straightforward evaluation of the sums leads to the desired result: $S_r(\chi)=0$. From now on, we thus  assume that $\chi$ is nontrivial: in which case, $\sum_{z'\neq 0}\chi(z')=0$ and $\chi(0)=0$, so that the last displayed identity reads:
\[S_r(\chi) = \chi^{-1}(4)\cdot \sum_{y\in\F_r} \chi(y+1) \left(\sum_{x\in\F_r} \chi(x) \lambda(x+y^2)\right).\]
Recall that $1+\lambda(w) = |\{t\in\F_r : t=w^2\}|$ for all $w\in\F_r$.
Thus, for any $y\in\F_r$, one can rewrite the sums over $x$ under the form:
\[
\sum_{x\in\F_r} \chi(x) \lambda(x+y^2) 
= \sum_{x\in\F_r} \chi(x)\cdot (1+ \lambda(x+y^2) ) 
= \sum_{x\in\F_r} \chi(x) \cdot \left|\{t\in\F_q: x=t^2-y^2\}\right|
= \sum_{t\in\F_r} \chi(t^2-y^2).\]
This leads to
\[ S_r(\chi) = \chi^{-1}(4) \cdot \sum_{(y,t)\in\F_r^2} 
\chi(t-y)\chi(t+y)\chi(y+1).\]
The map 
$ \F_r^2\to \ \left\{(x_1, x_2, x_3)\in\F_r^3 : x_1+x_2+x_3=1\right\}$ given by $(y,t)\mapsto\left(\frac{t-y}{2}, \frac{-(y+t)}{2}, y+1\right)$
is a bijection, and this allows us to write the latter double sum as a Jacobi sum: 
\[\sum_{(y,t)\in\F_r^2} \chi(t-y)\chi(t+y)\chi(y+1) 
= \chi(-4)\cdot\sum_{\substack{  x_1+x_2+ x_3=1, \\ x_i\in\F_r}} \chi(x_1)\chi(x_2)\chi(x_3) = \chi(-4) \cdot \ja_r(\chi, \chi, \chi).\]
Therefore, we have proved that $S_r(\chi) = \chi(-1)\cdot\ja_{r}(\chi, \chi, \chi)$. This concludes the proof.
\ProofEnd\end{proof}

\begin{prop}\label{prop.charsum.infty}
 Let $\F_r$ be a finite field of odd characteristic, and $a\in\F_r^\times$. 
Then 
\begin{equation}\label{eq.charsum.infty}
\sum_{x\in\F_r} \lambda(x^3+a) = - \lambda(a) \cdot 
\sum_{\substack{\xi^3=\trivcar \\ \xi \neq\trivcar}} \xi(4a)\cdot\ja_r(\xi, \xi,\xi),
\end{equation}
where the sum on the right-hand side is over nontrivial characters $\xi:\F_r^\times\to\Qbar^\times$ of order $3$.
\end{prop}

\begin{proof} 
The sum over $\xi$ in \eqref{eq.charsum.infty} contains $2$ or $0$ terms, depending on whether $3$ divides $|\F_r^\times|=r-1$ or not, respectively. 
For any $z\in\F_r$, one has $\left|\{x\in\F_{r} :  x=z^3 \}\right| = \sum_{\xi^3=\trivcar} \xi(z)$, where the sum is over characters $\xi$ on $\F_r^\times$ such that $\xi^3=\trivcar$ (see \cite[Lemma 2.5.21]{Cohen}). Therefore
\begin{equation}\label{eq.charsum.interm}
\sum_{x\in\F_{r}}\lambda(x^3+a)
= \sum_{z\in\F_{r}} \left|\{x\in\F_{r} :   x=z^3 \}\right|\cdot \lambda(z+a)
= \sum_{z\in\F_{r}}\lambda(z+a)  + \sum_{\substack{\xi^3=\trivcar \\ \xi\neq \trivcar}}\left(\sum_{z\in\F_{r}}\xi(z)\lambda(z+a)\right).\end{equation}
The first sum on the right-hand side vanishes because $\lambda$ is nontrivial.
If $3$ does not divide $|\F_{r}^\times|$, we are done since the sum over $\xi$ is empty. 
In the case where $3$ divides $|\F_{r}^\times|$, let $\xi$ be one of the two characters of exact order $3$ on $\F_{r}^\times$. Then,
\begin{align*}
\sum_{z\in\F_{r}}\xi(z)\lambda(z+a)
&= \sum_{\substack{x_1, x_2\in\F_r \\ x_1+x_2 = 1}} \xi(-ax_1)\lambda(ax_2) 
= \xi(-a)\lambda(a) \cdot\sum_{\substack{x_1, x_2\in\F_r \\ x_1+x_2 = 1}} \xi(x_1)\lambda(x_2) \\
&= \xi(-4a)\lambda(a) \cdot \sum_{\substack{x_1, x_2\in\F_r \\ x_1+x_2 = 1}} \xi(x_1)\xi(x_2) 
=-\xi(4a)\lambda(a)\cdot\ja_{r}(\xi, \xi, \xi).
\end{align*}
The penultimate equality follows from \cite[Prop. 2.5.18]{Cohen}, and the last one from \eqref{eq.jacobi.rel} because $\xi^3=\trivcar$. 
Plugging this result twice into \eqref{eq.charsum.interm} finishes the proof.
\ProofEnd\end{proof}

\subsection{Proof of Theorem \ref{theo.Lfunc}}
\label{sec.Lfunc.calc}

From the definition \eqref{eq.def.Lfunc} of the $L$-function, expanding $\log L(E_d/K, T)$ as a power series and rearranging terms as in \cite[\S 3.2]{Ulmer_legII}, one arrives at the following expression of $L(E_d/K, T)$:

\begin{lemm}
 Let $n\geq1$. For any $\tau\in\P^1(\F_{q^n})$, denote by $v_\tau$ the place of $K$ corresponding to $\tau$, and by $(\widetilde{E_d})_\tau$ the reduction of a integral minimal model of $E$ at $v_\tau$. 
 For all $n\geq 1$ and any $\tau\in\P^1(\F_{q^n})$, we let
$A_d(\tau, q^n) := q^n+1 - |(\widetilde{E_d})_\tau(\F_{q^n})|.$
Then, the $L$-function of $E_d/K$ satisfies the formal identity 
\begin{equation}\label{eq.practical.Lfunc}
\log L(E_d/K, T) = \sum_{n=1}^{\infty} \left(\sum_{\tau\in\P^1(\F_{q^n})}A_d(\tau, q^n)\right) \cdot \frac{T^n}{n}.
\end{equation}
\end{lemm}

Our first step will be to find a more explicit expression of the inner sums in \eqref{eq.practical.Lfunc}.
For any finite extension $\F_{q^n}$ of $\F_q$, we again denote by $\lambda :\F_{q^n}^\times\to\{\pm1\}$ 
(or by $\lambda_{q^n}$ when confusion is possible) 
the quadratic character on $\F_{q^n}^\times$. 
For any $\tau\in\P^1(\F_{q^n})$, we fix an affine model $y^2=f_\tau(x)$ of $(\widetilde{E_d})_\tau$ 
(with $f_\tau(x)\in\F_{q^n}[x]$ of degree~$3$). 
A standard computation yields that 
\begin{equation}\label{eq.Atau}
A_d(\tau, q^n)  
= q^n+1 - |\widetilde{(E_d)}_\tau(\F_{q^n})| 
= q^n - \sum_{x\in\F_{q^n}} \left(1 + \lambda\left( f_\tau(x)\right) \right) 
=- \sum_{x\in\F_{q^n}}  \lambda\left( f_\tau(x)\right).  
\end{equation}
The value of $A_d(\infty, q^n)$ depends on the reduction of $E_d$ at $\tau=\infty$  which, by Proposition \ref{prop.badred}, is as follows.  
When $d$ is not divisible by $3$,   $E_d$ has  additive reduction at $\infty$ and $(\widetilde{E_d})_{\infty}$ is a rational curve over $\F_{q}$, whence $A_d(\infty, q^n)= 0$ in that case. 
When $3$ divides $d$, $E_d$ has good reduction at $\infty$ and the reduced curve $\widetilde{(E_d)}_\infty$ has affine model $y^2=x^3+1/4$ 
over $\F_q$ (see Remark \ref{rem.minimod.loc}). Therefore, by \eqref{eq.Atau} 
and Proposition \ref{prop.charsum.infty} (with $r=q^n$ and $a=1/4$), one has
\[A_d(\infty, q^n) = -\sum_{x\in\F_{q^n}} \lambda(x^3+1/4) 
=\sum_{\substack{\xi^3=\trivcar \\ \xi\neq\trivcar}} \ja_{q^n}(\xi, \xi, \xi) 
= \sum_{\substack{\xi^3=\trivcar \\ \xi\neq\trivcar}} \xi(-1)\cdot\ja_{q^n}(\xi, \xi, \xi),\]
where the sum is over nontrivial characters $\xi$ of $\F_{q^n}^\times$ such that $\xi^3=\trivcar$ (note that $\xi(-1)=1$ for such a $\xi$).

Next for any $\tau\in\F_{q^n}$, as was noted in Remark \ref{rem.minimod.loc}, one can take $f_\tau(x) = x^3 + x^2 -8\tau^d \cdot x +16 \tau^{2d}$ and \eqref{eq.Atau} here leads to
\[
A_d(\tau, q^n)  
=- \sum_{x\in\F_{q^n}}  \lambda\left( x^3 + x^2 -8\tau^d \cdot x +16 \tau^{2d}\right).  
\]
For any $z\in\F_{q^n}$, one has $\left|\left\{\tau\in\F_{q^n} : \tau^d = z\right\}\right| = \sum_{\chi^d =\trivcar} \chi(z)$, where the sum is over all characters $\chi$ of $\F_{q^n}^\times$ such that $\chi^d=\trivcar$ (see \cite[Lemma 2.5.21]{Cohen}). 
After exchanging order of summation, we obtain that
\begin{align*}
\sum_{\tau \in\F_{q^n}} A(\tau, q^n) 
&= -\sum_{\tau\in\F_{q^n}} \sum_{x\in\F_{q^n}}  \lambda\left( x^3 + x^2 -8\tau^d \cdot x +16 \tau^{2d}\right) \\
&= -\sum_{z\in\F_{q^n}} \left(\left|\left\{\tau\in\F_{q^n} : \tau^d = z\right\}\right|\cdot\sum_{x\in\F_{q^n}}  \lambda\left( x^3 + x^2 -8z  x +16 z^2\right) \right)\\
&= - \sum_{\chi^d=\trivcar} \left( \sum_{z\in\F_{q^n}}   \sum_{x\in\F_{q^n}} \chi(z) \lambda\left( x^3 + x^2 -8z  x +16 z^2\right)\right).
\end{align*}
Using Proposition \ref{prop.charsum} on the inner sums (with $r=q^n$), we find that 
\[ \sum_{\tau \in\P^1(\F_{q^n})} A(\tau, q^n) 
= A(\infty, q^n)  - \sum_{\chi\in X(d, q^n)} \chi(-1)\cdot\ja_{q^n}(\chi, \chi, \chi),\]
where $X(d,q^n)$ is the set of nontrivial characters $\chi$ on $\F_{q^n}^\times$ such that $\chi^d=\trivcar$. 
By our expression of $A(\infty, q^n)$, 
it follows that
\begin{equation}\label{eq.pt.count}\notag
\sum_{\tau\in\P^1(\F_{q^n})} A_d(\tau, q^n) 
= \left\{ \begin{array}{ccl}
 \displaystyle\sum_{\substack{\xi^3=\trivcar \\ \xi\neq\trivcar}} \xi(-1)\cdot\ja_{q^n}(\xi, \xi, \xi) 
 &  \displaystyle- \sum_{\chi \in X(d, q^n)} \chi(-1)\cdot\ja_{q^n}(\chi, \chi, \chi)  &\text{if } 3\mid d, \\
 0 
 & \displaystyle - \sum_{\chi \in X(d, q^n)} \chi(-1)\cdot\ja_{q^n}(\chi, \chi, \chi)  
 &\text{else.}
 \end{array}
\right.
\end{equation}
In both cases, one can rewrite this as
\begin{equation}\notag
\sum_{\tau\in\P^1(\F_{q^n})} A_d(\tau, q^n) 
= - \sum_{\chi \in X_3(d, q^n)} \chi(-1)\cdot\ja_{q^n}(\chi, \chi, \chi), 
\end{equation}
where the sum is over the set $X_3(d, q^n)$ of nontrivial characters $\chi$ on $\F_{q^n}^\times$ such that $\chi^d=\trivcar$ and $\chi^3\neq \trivcar$ (see \S\ref{sec.char}).
Plugging this last identity 
into \eqref{eq.practical.Lfunc}, we obtain that
\begin{equation}\label{eq.proofL.step2}
- \log L(E_d/K, T) = \sum_{n\geq 1}\left( \sum_{\chi\in X_3(d, q^n)}  \chi(-1)\cdot\ja_{q^n}(\chi, \chi, \chi) \right) \cdot \frac{T^n}{n}.
\end{equation}

We now perform a ``reindexation'' on this double sum: Lemma \ref{lemm.reindex} allows us to rewrite \eqref{eq.proofL.step2} under the form:
\[\sum_{n\geq 1}\left( \sum_{\chi\in X_3(d,q^n)} \chi(-1)\cdot\ja_{q^n}(\chi, \chi, \chi) \right) \cdot \frac{T^n}{n} = \sum_{m\in Z_d} \left(\sum_{s\geq 1}    \  \tt_m^{(s)}(-1)\cdot\ja_{q^{s\cdot |\Mm|}}(\tt_m^{(s)}, \tt_m^{(s)}, \tt_m^{(s)}) \cdot \frac{T^{s\cdot |\Mm|}}{s\cdot |\Mm|}\right).\]
Further, $\tt_m^{(s)}(-1) = \tt_m(-1)^s$ and the Hasse-Davenport relation \eqref{eq.jacobi.HD} 
implies that
\[\forall m\in Z_d, \forall s\geq 1, \quad \ja_{q^{s\cdot |\Mm|}}(\tt_m^{(s)}, \tt_m^{(s)}, \tt_m^{(s)}) =  \ja_{q^{ |\Mm|}}(\tt_m, \tt_m, \tt_m)^{s} = \Ja(m)^s.\]
Therefore, we derive that
\begin{align*}
\sum_{n\geq 1}\left( \sum_{\chi\in X_3(d,q^n)} \chi(-1)\cdot\ja_{q^n}(\chi, \chi, \chi) \right) \cdot \frac{T^n}{n} 
&= \sum_{m\in Z_d} \sum_{s\geq 1}     \frac{\left(\tt_m(-1)\Ja(m)\cdot T^{ |\Mm|} \right)^s}{s\cdot |\Mm|} \\
&= - \sum_{m\in Z_d} \frac{ 1}{|\Mm|}\cdot \log\left(1-\tt_m(-1)\Ja(m)\cdot T^{|\Mm|}\right).
\end{align*}
In the right-most sum, notice that each term ``$\log\left(1-\tt_m(-1)\Ja(m)\cdot T^{|\Mm|}\right)$'' appears $|\Mm|$ times  since ${\Ja(q^k\cdot m) = \Ja(m)}$ for all $k\geq 1$. 
Thanks to the ``weighting'' by $1/|\Mm|$, we may thus write the sum over $m\in Z_d$ as a sum over $\Mm\in\O_q(Z_d)$.
Finally, we have proved
\[\log L(E_d/K, T) = \sum_{\Mm\in\O_q(Z_d)} \log\left(1- \tt_m(-1)\Ja(m)\cdot T^{|\Mm|}\right).\]
It remains to exponentiate this identity to complete the proof of Theorem \ref{theo.Lfunc}.
\ProofEnd

\section{Bounds on the special value}
\label{sec.spval.bnd}

We now  study in more detail the behaviour of $L(E_d/K, T)$ around the point $T=q^{-1}$. 
More specifically, recall that $\rho(E_d/K) := \ord_{T=q^{-1}} L(E_d/K, T)$ is called the \emph{analytic rank} of $E_d$, 
and  that  the \emph{special value of $L(E_d/K, T)$ at $T=q^{-1}$} is  the quantity  
\begin{equation}\label{eq.def.spval}
L^\ast(E_d/K, 1) := \left. \frac{L(E_d/K, T)}{(1-qT)^\rho}\right|_{T=q^{-1}} \in \Z[q^{-1}]\smallsetminus\{0\} 
\quad \text{ where } \rho=\rho(E_d/K).
 \end{equation}
 
\begin{rema} 
The special value is ``usually'' defined as the first nonzero coefficient in the Taylor expansion of $s\mapsto L(E_d/K, q^{-s})$  around $s=1$: our definition \eqref{eq.def.spval} differs from this more ``usual'' one by a factor $(\log q)^\rho$, which we prefer to avoid in order to ensure that $L^\ast(E_d/K, 1)\in\Q^\ast$. 
Note that this is consistent with our normalization of $\Reg(E_d/K)$ (see \S\ref{sec.BS} below).
\end{rema}

The goal of this section is to give an asymptotic  estimate on the size of $L^\ast(E_d/K, 1)$ in terms of the height $H(E_d/K)$, as $d$ grows to $+\infty$. By a rather crude estimate, as in \cite[\S7]{HP15} for example, one readily obtains that 
\[ - 5 +o(1)
\leq \frac{\log L^\ast(E_d/K, 1)}{\log H(E_d/K)} 
\leq o(1) \qquad (\text{as }d\to\infty).
\]
Here, we prove the following improved bounds:
  
 \begin{theo}\label{theo.spval.Bnd} 
For all $\epsilon\in(0,1/4)$, there exist positive constants $C_1, C_2$, depending at most on $p$, $q$ and~$\epsilon$, such that  for any integer $d\geq 2$ coprime to $q$, the special value $L^\ast(E_d/K, 1)$ satisfies: 
\begin{equation}\label{eq.spval.LBnd.ccl}
- C_1 \cdot   \left(\frac{\log \log d}{\log d}\right)^{1/4-\epsilon}
\leq \frac{\log L^\ast(E_d/K, 1)}{\log H(E_d/K)} 
\leq C_2  \cdot  \frac{\log \log d}{\log d}.
\end{equation}
 \end{theo}

In the next section, we will use the BSD conjecture to reveal the arithmetic significance of such an estimate.
The rest of this section is dedicated to the proof of Theorem \ref{theo.spval.Bnd}. 
Using Theorem \ref{theo.Lfunc}, we notice that $L(E_d/K, T)$ is a polynomial of the type studied in \cite{Griffon_FermatS}.
The desired bounds on $L^\ast(E_d/K, 1)$ are then a direct consequence of the results of \emph{loc. cit.}, which we start by recalling.

\subsection{Framework for bounding some special values}
\label{sec.framework}

For the convenience of the reader, we briefly recall the setting introduced in \cite[\S3]{Griffon_FermatS} to prove bounds on special values of polynomials of the type appearing in \eqref{eq.Lfunc}. 
For any integer $d\geq 2$ coprime to $q$, consider 
\[G_d := \left\{\ba=(a_0, a_1, a_2, a_3)\in(\Z/d\Z)^4 : a_0 + a_1+a_2+a_3=0\right\}.\] 
The group $(\Z/d\Z)^\times$ acts on $G_d$ by coordinate-wise multiplication. In particular, $q$ acts by multiplication on $G_d$ and, for any nonempty subset $\Lambda\subset G_d$ which is stable under the action of $q$, we denote by $\O_q(\Lambda)$ the set of orbits.
For $\ba\in G_d$, we denote its orbit by $\bA =\{\ba, q \ba, q^2 \ba, \dots\}$.
We say that a nonempty subset $\Lambda\subset G_d$ satisfies hypothesis \eqref{eq.hyp} if for all $\epsilon\in(0,1/4)$, there exists $u\in(0,1)$ such that
\begin{equation}\label{eq.hyp}\tag{H}
 \left| \left\{ (a_0, \dots, a_3)\in\Lambda : \ d>\max_{0\leq i\leq 3}\{\gcd(d,a_i)\} > d^u\right\}\right| \leq c' \cdot {|\Lambda|}\cdot \left(\frac{\log\log d}{\log d}\right)^{1/4-\epsilon},
\end{equation}
for some constant $c'$.

For $\ba=(a_0, \dots, a_3)\in G_d$ with $\ba\neq\bm{0}=(0,0,0,0)$, whose orbit $\bA$ has length $|\bA|$, we define  four characters on $\F_{q^{|\bA|}}^\times$:
\[\forall i\in\{0, 1,2,3\}, \quad 
 \cchi_{i}: \F_{q^{|\bA|}}^\times \to \Qbar^\times, \ x \mapsto \tt(x)^{(q^{|\bA|}-1)\cdot a_i/d}.
\]
One then defines a Jacobi sum:
\[ \Ja'(a_0, a_1, a_2, a_3) = \frac{1}{{q^{|\bA|}}-1}\sum_{\substack{x_0, \dots, x_3 \in\F_{q^{|\bA|}}^\times\\ x_0 +\dots + x_3=0}}\cchi_0(x_0)\cchi_1(x_1)\cchi_2(x_2)\cchi_3(x_3) \in\Q(\zeta_d).\]
Since $\ba\in G_d$, the product $\cchi_0\cchi_1\cchi_2\cchi_3$ is the trivial character on $\F_{q^{|\bA|}}^\times$, and a classical calculation (\cf{} Lemma 2.5.13 in \cite{Cohen}) relates $\Ja'(a_0, a_1, a_2, a_3)$ to the Jacobi sums in Definition \ref{defi.jacobi}: 
\begin{equation}\label{eq.jacobi.link}
\Ja'(a_0, a_1, a_2, a_3) 
=  (\cchi_0\cchi_1\cchi_2)(-1)\cdot \ja_{q^{|\bA|}}(\cchi_0, \cchi_1, \cchi_2).
\end{equation}
For any $\ba=(a_0, a_1, a_2, a_3)\in G_d\smallsetminus\{\bm{0}\}$, it is well-known that $\Ja'(\ba)=0$ as soon as some (but not all) of the $a_i$'s are $0\bmod{d}$, and that $|\Ja'(\ba)|=q^{|\bA|}$ if all $a_i$'s are nonzero.

To any nonempty  subset $\Lambda\subset G_d$ which is stable under the action of $(\Z/d\Z)^\times$, we can associate a polynomial
\[P(\Lambda, T) := \prod_{\bA\in\O_q(\Lambda)} \left(1-\Ja'(\ba)  \cdot T^{|\bA|}\right),\]
where, for any orbit $\bA$, $\ba\in G_d$ denotes a choice of representative of $\bA$. 
Since the action of $\Gal(\Q(\zeta_d)/\Q)$ on $\{\Ja'(\ba)\}_{\ba\in G_d}$ corresponds to the action of $(\Z/d\Z)^\times$ on $G_d$ in the isomorphism $\Gal(\Q(\zeta_d)/\Q)\simeq (\Z/d\Z)^\times$, the assumption that $\Lambda$ is $(\Z/d\Z)^\times$-stable ensures that  $P(\Lambda, T)\in\Z[T]$. 
For $\Lambda$ as above, we finally introduce a 
\emph{special value}:
\begin{equation}\label{eq.def.spval.lambda}
P^\ast(\Lambda) := \left.\frac{P(\Lambda, T)}{(1-qT)^\rho}\right|_{T=q^{-1}} \in \Z[q^{-1}]\smallsetminus\{0\} 
\quad \text{ where } \rho=\ord_{T=q^{-1}} P(\Lambda, T).\end{equation}

The following statement summarizes the main technical results of \cite{Griffon_FermatS}: 
\begin{theo}\label{theo.bnd.spval}
For all $\epsilon\in(0,1/4)$, there exist positive constants $C_3, C_4$,  depending at most on $q$, $p$ and $\epsilon$, such that the following holds.
For any integer $d\geq 2$ coprime to $q$, and any nonempty subset $\Lambda\subset G_d$  which is stable under the action of $(\Z/d\Z)^\times$ and for which hypothesis \eqref{eq.hyp} holds, the special value $P^\ast(\Lambda)$ satisfies
\begin{equation}\label{eq.conclu2}
 - C_3 \cdot\left(\frac{\log\log d}{\log d}\right)^{1/4-\epsilon}\leq \frac{\log |P^\ast(\Lambda)|}{  \log q^{|\Lambda|} } 
 \leq C_4 \cdot\frac{\log\log |\Lambda|}{\log|\Lambda|}.
 \end{equation}
\end{theo} 
This theorem is a concatenation of  Theorems 5.1 and 6.2 in 
\emph{loc.cit.}:  the proof of the upper bound is relatively straighforward but  that of the lower bound is more delicate. 
It essentially involves two ingredients: the Stickelberger theorem about $p$-adic valuations of Jacobi sums  and an average equidistribution theorem for subgroups of $(\Z/d\Z)^\times$ (see \cite[\S4, \S6]{Griffon_FermatS} for more details).

\subsection{Proof of Theorem \ref{theo.spval.Bnd}} 
\label{sec.bnd.spval}

In order to apply Theorem \ref{theo.bnd.spval} to the special value $L^\ast(E_d/K, 1)$, we start by relating $L(E_d/K, T)$ to a certain
$P(\Lambda, T)$ as in the last subsection. 
Namely, for any integer $d\geq 2$ coprime to $q$, we consider the subgroup $H_d\subset G_d$ generated by $\bm{u}=(1, 1, 1, -3)$ and we let
\begin{equation}\label{eq.def.Lambda}
\Lambda_d := H_d\smallsetminus\{\bm{0}\} 
=  \left\{(m,m,m,-3m), \ m\in \Z/d\Z\right\}\smallsetminus\{\bm{0}\}.
\end{equation}
Being  a subgroup of $G_d$, $H_d$ is nonempty and stable under multiplication by $(\Z/d\Z)^\times$, and so is $\Lambda_d$. 
We clearly have $|\Lambda_d| = |\Z/d\Z| - 1 =d-1$.
Let $\ba=  (a_0, \dots, a_3)$  be an element of $\Lambda_d$, so that $\ba = m \cdot \bm{u}$ for some $m\in\Z/d\Z\smallsetminus\{0\}$. Notice first that $|\bA| = |\Mm|$ because the coordinates of $\bm{u}$ are pairwise coprime. Also, all the $a_i$ are nonzero 
if and only if $m\in Z_d$.
Furthermore, it follows from \eqref{eq.jacobi.link} that
\[\Ja'(\ba) 
= \Ja'(m,m,m,-3m)
= \tt_m(-1)^3\cdot \ja_{q^{|\Mm|}}(\tt_m, \tt_m, \tt_m)
= \tt_m(-1)\cdot \Ja(m).\]
Consequently, 
in the notations of \S\ref{sec.framework}, Theorem \ref{theo.Lfunc} translates  as: 
 
\begin{coro}\label{coro.Lfunc.link}
 Let $d\geq 2$ be an integer coprime to $q$, define $\Lambda_d$ as in \eqref{eq.def.Lambda}. Then the $L$-function of $E_d$ is given by
\[L(E_d/K, T) = P(\Lambda_d, T)\in\Z[T].\]
\end{coro}
In particular, since the definitions \eqref{eq.def.spval} and \eqref{eq.def.spval.lambda} of special values agree, we see that 
\[L^\ast(E_d/K, 1)=P^\ast(\Lambda_d).\]

Let us now check that the subset $\Lambda_d\subset G_d$ defined in \eqref{eq.def.Lambda} satisfies (a strong form of) hypothesis \eqref{eq.hyp}:
 \begin{lemm} 
 For all $u\in(0,1)$, one has
\[ \left| \left\{\ba\in\Lambda_d : \ d>\max_i\{\gcd(d,a_i)\} > d^u\right\}\right| \ll_u d^{-u/2} \cdot |\Lambda_d|.
\]
\end{lemm}

\begin{proof} By construction of $H_d$, any $\ba=(a_0, \dots, a_3)\in \Lambda_d$ is of the form $\ba = (m, m, m, -3m)$ for some ${m\in\Z/d\Z\smallsetminus\{0\}}$. 
Thus, one has 
 $\max_i\{\gcd(d,a_i)\}\leq 3\gcd(d, m)$.
In particular, we obtain that 
\[ \left| \left\{\ba\in\Lambda_d : \ d>\max_i\{\gcd(d,a_i)\} > d^u\right\}\right| 
\leq \left| \left\{m\in\Z/d\Z\smallsetminus\{0\} : \  \gcd(d,m)> {d^u}/{3}\right\}\right|. \]
For any divisor $e$ of $d$, the number of $m\in\Z/d\Z$ such that $\gcd(d,m)=e$ is $\phi(d/e)\leq d/e$. Hence,  
\begin{align*}
\left| \left\{ m\in\Z/d\Z: \  \gcd(d,m)> d^u/3\right\}\right|
&= \sum_{ \substack{e\mid d \\ d^u/3<e }} \left| \left\{ m\in\Z/d\Z :  \  \gcd(d,m)=e\right\}\right| 
\leq  \sum_{ \substack{e\mid d \\ d^u/3<e }}  \frac{d}{e} \leq \frac{3 \tau(d) \cdot d}{d^u},
\end{align*}
where $\tau(d)$ is the number of divisors of $d$.
By a classical theorem, for all $v>0$, there is an explicit constant $c_v$ such that $\tau(d) \leq c_v\cdot d^v$  (see \cite[Thm. 315]{HardyWright}). 
In particular, for $v=u/2 >0$, we have:
\[
 \left| \left\{\ba\in\Lambda_d : \ d>\max_i\{\gcd(d,a_i)\} > d^u\right\}\right| \leq 3 \tau(d)d^{-u} \cdot |\Lambda_d| 
\ll_ u d^{-u/2}\cdot |\Lambda_d|.
\]
This proves the Lemma, and shows that $\Lambda_d$ satisfies hypothesis \eqref{eq.hyp}.
\ProofEnd

\paragraph{}
Together with Corollary \ref{coro.Lfunc.link},   the previous Lemma implies that Theorem \ref{theo.bnd.spval} applies to ${P^\ast(\Lambda_d)=L^\ast(E_d/K, 1)}$.
Remembering that $|\Lambda_d|=d-1$,  we obtain that 
\[  - C_3 \cdot\left(\frac{\log\log d}{\log d}\right)^{1/4-\epsilon}\leq \frac{\log L^\ast(E_d/K, 1)}{  \log q^{d-1} } 
 \leq C_4 \cdot\frac{\log\log d}{\log d}, \quad (\text{as } d\to\infty).\]
By \eqref{eq.inv2}, we have 
\[ \forall d\geq 2, \qquad 
  \frac{9}{4} \leq \frac{3(d+1)}{d+2}
\leq \frac{\log q^{d-1}}{\log H(E_d/K)} = \frac{d-1}{\partint{(d+2)/3}}
\leq \frac{3(d+1)}{d-1}\leq 9.
\]
Combining the last two displayed sets of inequalities concludes the proof of Theorem \ref{theo.spval.Bnd}. \ProofEnd \end{proof}

\section{Analogue of the Brauer-Siegel theorem}
\label{sec.BS}

In this section, we reinterpret the bounds in  Theorem \ref{theo.bnd.spval} in terms of arithmetic invariants of $E_d/K$, which we first introduce.
By the analogue of the Mordell--Weil theorem for elliptic curves over $K$, the group $E_d(K)$ is finitely generated (\cf{} \cite[Lect. 1, Thm. 5.1]{UlmerParkCity}). 
Furthermore, the group $E_d(K)$ is endowed with the canonical  Néron--Tate height $\hhat{h}_{NT}:E_d(K) \to \Q$. 
The quadratic map $\hhat{h}_{NT}$ induces a $\Z$-bilinear pairing $\langle -, - \rangle_{NT}: E_d(K)\times E_d(K)\to\Q$, which is nondegenerate modulo $E_d(K)\tors$ (\cf{} \cite[Chap. III, Thm. 4.3]{ATAEC}).
The \emph{Néron-Tate regulator} of $E_d/K$ is then defined as: 
\[\Reg(E_d/K) := \left|\det\left( \langle P_i, P_j \rangle_{NT}\right)_{1\leq i,j \leq r}\right| \in\Q^\ast,\]
for any choice of a $\Z$-basis $P_1, \dots, P_r \in E_d(K)$ of $E_d(K)/E_d(K)\tors$. 
Note that we normalize $\langle -, - \rangle_{NT}$ to have values in $\Q$: we may do so since, in our context, this height pairing has an interpretation as an intersection pairing on the minimal regular model of $E_d$ (see \cite[Chap. III, \S9]{ATAEC}).    
Let us also recall that the \emph{Tate--Shafarevich group} of $E_d/K$
is defined by
\[\sha(E_d/K) :=
\ker\left( \H^1(K, E_d) \longrightarrow \prod_{v} \H^1(K_v, (E_d)_v)\right),\]
see \cite[Lect. 1, \S11]{UlmerParkCity} for more details. 
In Theorem \ref{theo.BSD} below, we will see that $\sha(E_d/K)$ is finite.

\subsection{The BSD conjecture}\label{sec.BSD}

Inspired by the conjecture of Birch and Swinnerton-Dyer for elliptic curves over $\Q$, Tate gave a conjectural arithmetic interpretation of $\rho(E_d/K)$ and $L^\ast(E_d/K, 1)$ (see \cite{Tate_BSD}). The conjecture is still open in general, but has been proved in the case of $E_d$ by Ulmer. We state his result as follows:

\begin{theo}[Ulmer]\label{theo.BSD}
For all integers $d\geq 1$, coprime with $q$, let $E_d$ be the Hessian elliptic curve \eqref{eq.Wmodel} as above. Then the full BSD conjecture is true for $E_d/K$: that is to say,
\begin{enumerate}[$\bullet$]
\item the Tate-Shafarevich group $\sha(E_d/K)$ is finite,
\item the rank of $E_d(K)$ is equal to $\ord_{T=q^{-1}} L(E_d/K, T)$,
\item moreover, one has
\begin{equation}\label{eq.BSD}
L^\ast(E_d/K, 1) = \frac{|\sha(E_d/K)|\cdot \Reg(E_d/K)}{H(E_d/K)}\cdot \frac{\tam(E_d/K)\cdot q}{|E_d(K)\tors|^2}.
\end{equation}
\end{enumerate}
\end{theo}
We refer the reader to \cite[\S6]{Ulmer_largeLrank} for the proof, or to \cite[Lect. 3, \S10]{UlmerParkCity} for a detailed sketch.

\begin{rema} Given Corollary \ref{coro.Lfunc.link}, Lemma 3.5 in \cite{Griffon_FermatS} yields fairly explicit expressions for $\rho(E_d/K)$ and $L^\ast(E_d/K, 1)$ as follows. 
For any integer $d\geq 2$, in the notations of \S\ref{sec.prel}, consider the two subsets of $Z_d$ given by
\[ V_d := \left\{ m \in Z_d : \ \tt_m(-1)\Ja(m)=q^{|\Mm|}\right\} \quad \text{ and } \quad
S_d :=Z_d\smallsetminus V_d. \]
It is easy to check that the sets $V_d$ and $S_d$ are stable under multiplication by $q$. Then,  the analytic rank  is given by $\rho(E_d/K)=|\O_q(V_d)|$, and the special value $L^\ast(E_d/K, 1)$ has the following expression:
\begin{equation}\label{eq.expr.spval}
L^\ast(E_d/K, 1) = \prod_{\Mm \in \O_q(V_d)} |\Mm| \cdot \prod_{\Mm\in\O_q(S_d)} \left(1- \tt_m(-1)\Ja(m) \cdot q^{-|\Mm|}\right).
\end{equation}

\end{rema}

\begin{rema}  
By Theorem \ref{theo.BSD} above, the analytic rank $\rho(E_d/K)$ is equal to $\mathrm{rank}(E_d(K))$. 
The expression for $\rho(E_d/K)$ obtained in the previous remark allows us to retrieve a result of Ulmer on the ranks of $E_d(K)$
stating that  as $d$ ranges though integers coprime to $q$, the ranks of $E_d(K)$ are unbounded (see \cite[\S2-\S4]{Ulmer_largeLrank},  \cite[Lect. 4, Thm. 3.1.1]{UlmerParkCity}). 
More precisely, one can show that there are infinitely many integers $d'\geq 2$ coprime to $q$, such that 
$\mathrm{rank}(E_{d'}(K))\gg_q { {d'} }/{\log d'},$
where the implied constant is effective and depends only on $q$. 
We refer to \cite[Prop. 7.3.5]{Griffon_PHD} for more details. 
\end{rema}

We conclude this subsection by recording the following estimate (see also \S2 in \cite{HP15}): 

\begin{coro}\label{coro.link.spval.BS}
When $d\to\infty$ runs over the integers coprime to $q$, one has:
\[ \frac{\log\big(|\sha(E_d/K)|\cdot\Reg(E_d/K)\big)}{\log H(E_d/K)} = 1 + \frac{\log L^\ast(E_d/K, 1) }{\log H(E_d/K)} + O\left(\frac{\log d}{d}\right),\]
where the implicit constant is effective and depends at most on $q$.
\end{coro}
\begin{proof} We first note that Theorem \ref{theo.BSD} ensures that $\sha(E_d/K)$ is a finite group, so that the quantity on the left-hand side makes sense. For any integer $d\geq 2$ coprime to $q$, 
we take the logarithm of \eqref{eq.BSD} and divide throughout by $\log H(E_d/K)$. Reordering terms, we obtain that 
\[\frac{\log\big(|\sha(E_d/K)|\cdot\Reg(E_d/K)\big)}{\log H(E_d/K)} = 1 + \frac{\log L^\ast(E_d/K, 1) }{\log H(E_d/K)} + \frac{\log \left( \tam(E_d/K)\cdot q \cdot |E_d(K)\tors|^{-2}\right)}{\log H(E_d/K)}.\]
Corollary \ref{coro.tamtors.Bnd} then allows us to control the size of the right-most term. This yields the desired result.
\ProofEnd
\end{proof}

\subsection{Analogue of the Brauer-Siegel theorem}

 We finally turn to the proof of the asymptotic estimate announced in Theorem \ref{theo.i.BS} of the introduction:

\begin{theo}\label{theo.BS} When $d\geq 2$ ranges through integers coprime to $q$,   one has the asymptotic estimate:
\begin{equation}\label{eq.BS}
\log\big( |\sha(E_d/K)|\cdot \Reg(E_d/K)\big) \sim \log H(E_d/K) \qquad(\text{as } d\to\infty). 
\end{equation}
\end{theo}

\begin{proof}
 Given what has already been proved, the proof is almost clear: by Corollary \ref{coro.link.spval.BS}, we know that 
\[ \frac{\log\big(|\sha(E_d/K)|\cdot\Reg(E_d/K)\big)}{\log H(E_d/K)} 
= 1 + \frac{\log L^\ast(E_d/K, 1) }{\log H(E_d/K)} + O\left(\frac{\log d}{d}\right)
 \qquad (\text{as } d\to\infty).\]
Further, Theorem \ref{theo.bnd.spval} implies that, for all $\epsilon\in(0,1/4)$,  there exists a constant $C_5>0$ such that
\[ \left|  \frac{\log L^\ast(E_d/K, 1) }{\log H(E_d/K)}  \right| 
\leq C_5 \cdot \left(\frac{\log \log d}{\log d}\right)^{1/4-\epsilon} 
\qquad (\text{as } d\to\infty).\]
The concatenation of these two results thus yields that:
\[ \frac{\log\big(|\sha(E_d/K)|\cdot\Reg(E_d/K)\big)}{\log H(E_d/K)} = 1 +O\left(\left(\frac{\log \log d}{\log d}\right)^{1/4-\epsilon}\right) = 1+ o(1)  \qquad (\text{as }d\to\infty),\] 
where the implicit constant in the intermediate equality is effective and depends at most on $q$, $p$ and $\epsilon$. This is more than enough to prove Theorem \ref{theo.BS}.
\ProofEnd \end{proof}

\begin{rema}
In \cite{HP15}, Hindry and Pacheco 
suggested to investigate the asymptotic behaviour 
of the \emph{Brauer-Siegel ratio}
\[\BS(E/K) :=  \log\big( |\sha(E/K)|\cdot \Reg(E/K)\big)\big/\log H(E/K), \]
as $E$ runs through a family of non-isotrivial elliptic curves over $K$.
If $\ELL$ denotes the family of all such 
elliptic curves 
ordered by differential height, 
\cite[Coro. 1.13]{HP15} proves that 
\[ \textstyle 0\leq \liminf_{E\in\ELL} \BS(E/K) \leq \limsup_{E\in\ELL} \BS(E/K) = 1,\]
conditionally to the BSD conjecture for all $E\in\ELL$. 
In this terminology, Theorem \ref{theo.BS} above can be rephrased as follows: the   ratio $\BS(E_d/K)$ has a limit when $E_d$ ranges through the Hessian family of elliptic curves (with $d\to\infty$),
and this limit is $1$ (unconditionally).
\end{rema}

\paragraph{Acknowledgements} 
This article is based on results obtained  by the author in his PhD thesis \cite{Griffon_PHD}. 
He thanks his supervisor Marc Hindry for his guidance and his 
encouragements.
He also wishes to thank Douglas Ulmer and Michael Tsfasman for their interest in this work, for carefully reading a previous version, and for their helpful comments. 
The author is grateful to Universiteit Leiden for providing 
great working conditions during the writing of this article.


\newcommand{\mapolicebackref}[1]{\hspace*{-5pt}{\textcolor{olive!50}{\small$\uparrow$ #1}}}
\renewcommand*{\backref}[1]{
\mapolicebackref{#1}}
\hypersetup{linkcolor=olive!50}
\pdfbookmark[0]{References}{references} 
\addcontentsline{toc}{section}{References}

\bibliographystyle{alpha}
\bibliography{../../Biblio_GENERAL.bib}

\end{document}